\documentclass{article}

\usepackage[T1]{fontenc}
\usepackage[utf8]{inputenc}
\usepackage[english]{babel}
\usepackage[a4paper]{geometry}
\usepackage[osf,sc]{mathpazo}
\usepackage{amsmath}
\usepackage{amssymb}
\usepackage{tikz-cd}

\usepackage{titlesec}
\titleformat{\section}{\normalfont\large\bfseries}{\thesection.}{.5em}{}

\usepackage[amsmath,thmmarks]{ntheorem}
\numberwithin{equation}{section}
\theoremseparator{.}
\newtheorem{theorem}[equation]{Theorem}
\newtheorem{lemma}[equation]{Lemma}
\theorembodyfont{\normalfont}
\newtheorem{definition}[equation]{Definition}
\theoremsymbol{\ensuremath{\blacktriangleleft}}
\newtheorem{remark}[equation]{Remark}
\theoremstyle{nonumberplain}
\theoremheaderfont{\itshape}
\theoremsymbol{\ensuremath{\blacksquare}}
\newtheorem{proof}{Proof}

\usepackage{enumitem}
\setlist{nosep,leftmargin=\parindent,listparindent=\parindent}
\setlist[itemize]{label=$\vcenter{\hbox{\small$\bullet$}}$}

\renewcommand{\AA}       {{\mathbb{A}}}
\newcommand  {\Aut}      {\mathop{\textup{Aut}}}
\newcommand  {\AUT}      {\mathop{\mathcal{A}\textit{ut}}}
\newcommand  {\et}       {{\textup{et}}}
\newcommand  {\F}        {{\mathcal{F}}}
\newcommand  {\ftimes}[1]{\times_{\!#1}}
\newcommand  {\HH}       {\mathop{\textup{H}}}
\newcommand  {\id}       {{\textup{id}}}
\newcommand  {\Jac}      {\mathop{\textup{Jac}}}
\newcommand  {\M}        {{\mathcal{M}}}
\newcommand  {\PP}       {{\mathbb{P}}}
\newcommand  {\Sch}      {{\textup{Sch}}}
\newcommand  {\Spec}     {\mathop{\textup{Spec}}}
\newcommand  {\U}        {{\mathfrak{U}}}

\begin{document}

\vspace*{2em}
\begin{center}
	{\LARGE\bfseries Ineffective descent of genus one curves} \\[1em]
	Wouter Zomervrucht \\
	\today
\end{center}
\vspace{2em}

\begin{quote}
	\small
	{\bfseries Abstract.} Raynaud proved in 1968 that étale descent of genus one curves is not effective in general. In this paper we provide an alternative, simplified construction of this phenomenon. Our counterexample is fully explicit.
\end{quote}
\vspace{.5em}

\section{Introduction}
\label{sec:intro}

Let $\U = \{U_i : i \in I\}$ be a cover of a scheme $S$, in some topology. A \emph{descent datum of schemes} relative to $\U$ consists of schemes $X_i$ over $U_i$ for all $i \in I$ and isomorphisms $\phi_{ji} \colon X_i \ftimes{U_i} U_{ij} \to X_j \ftimes{U_j} U_{ij}$ over $U_{ij}$ for all $i,j \in I$, satisfying the cocycle condition $\phi_{ki} = \phi_{kj} \phi_{ji}$ on $U_{ijk}$.

In favorable situations the descent datum is \emph{effective}, i.e. descends to a scheme $X$ over $S$. For instance, if $\U$ is a Zariski open cover, descent data are better known as gluing data and always effective. In larger topologies, such as the étale topology, ineffective descent data occur. An example can be found in \cite[03FN]{bib:stacks}.

A natural next step is to determine classes of schemes for which étale descent is effective. In this paper we consider the case of genus one curves. First, let us once and for all fix our notion of a (relative) curve.

\begin{definition}
\label{def:curve}
	Let $S$ be a scheme. A \emph{curve of genus $g$} over $S$ is a proper smooth scheme $X/S$ of relative dimension one, whose geometric fibers are connected curves of genus $g$.
\end{definition}
For genus one curves, the result is negative.

\begin{theorem}
\label{thm:genus_one}
	There exist ineffective étale descent data of genus one curves.
\end{theorem}
This theorem is due to Raynaud \cite[XIII 3.2]{bib:raynaud}. The aim of the current paper is to provide a simplified counterexample.

In the case of genus $g \neq 1$ curves, it is well-known that étale, or even fpqc, descent is effective. Indeed, outside genus one the canonical bundle (or its dual) is ample, and the descent comes from descent of quasi-coherent sheaves. See e.g. \cite[4.39]{bib:vistoli} for more details. On genus one curves the canonical bundle is fiberwise trivial and the argument does not apply. Note however that for \emph{elliptic curves} the zero section provides an ample line bundle, and fpqc descent is again effective.

Theorem \ref{thm:genus_one} can also be interpreted as follows. Let $\F$ be the fibered category over $\Sch$ that assigns to a scheme $S$ the groupoid $\F(S)$ of genus one curves over $S$. Then $\F$ is not an étale stack. Instead one works with the fibered category $\M_1$ where $\M_1(S)$ is the groupoid of proper smooth algebraic spaces over $S$ whose geometric fibers are genus one curves. Now $\M_1$ is an étale (even fppf) stack; it is the fppf stackification of $\F$.

\medskip \noindent
\textbf{Organization.}
The next section reduces the proof of theorem \ref{thm:genus_one} to one concerning torsors under elliptic curves. Section \ref{sec:construction} contains the actual construction. In section \ref{sec:raynaud} we compare it with the original counterexample by Raynaud.

\medskip \noindent
\textbf{Acknowledgements.}
The research in this paper is part of the author's master's thesis. I would like to thank Lenny Taelman and Bas Edixhoven for their valuable contributions.

\section{Torsors}
\label{sec:torsors}

Attached to a genus one curve $X/S$ is its Jacobian $E = \Jac(X/S)$. It is an elliptic curve over $S$, endowed with a natural action on $X$ that makes $X$ into an étale $E$-torsor.

Now let $\U$ be an étale cover of $S$. A descent datum of genus one curves relative to $\U$ descends to a (not necessarily representable) sheaf of sets $X$ on $S$. Also, by functoriality of $\Jac$ we obtain a descent datum of elliptic curves relative to $\U$. The latter descends to an elliptic curve $E/S$. Again the natural action of $E$ on $X$ makes $X$ into an étale $E$-torsor.

Conversely, any étale $E$-torsor gives rise to a descent datum of genus one curves relative to some étale cover. So the problem at hand is really to find a non-representable torsor under some elliptic curve. The following theorem from \cite[XIII 2.6]{bib:raynaud} will be useful.

\begin{theorem}
\label{thm:raynaud}
	Let $S$ be a local scheme and $E/S$ an elliptic curve.
	\begin{itemize}
		\item If $S$ is normal, an étale $E$-torsor is representable if and only if it has finite order in $\HH^1(S,E)$.
		\item If $S$ is regular, all étale $E$-torsors are representable.
	\end{itemize}
\end{theorem}
Here and further on, $\HH^1(S,E)$ always denotes sheaf cohomology on $(\Sch/S)_\et$. Recall that \emph{normal} means all local rings of $S$ are integrally closed domains.

\section{Construction}
\label{sec:construction}

We shall now construct a noetherian normal local scheme $S$, an elliptic curve $\tilde{E}/S$, and a class $\tau \in \HH^1(S,\tilde{E})$ of infinite order. Necessarily $S$ is irregular, so not of dimension $0$ or $1$.

Let $k$ be a field of characteristic not $2$. Let $E/k$ be an elliptic curve with a given embedding in $\PP^2$. Let $C \subset \AA^3$ be the affine cone over $E$. Let $s \in C$ be the top of the cone, and $S$ the localization of $C$ at $s$. It is normal by Serre's criterion \cite[23.8]{bib:matsumura} since $S$ is a complete intersection with its singularity in codimension $2$.

\begin{lemma}
\label{lem:cover}
	There exists a finite étale cover $\pi \colon S' \to S$ of degree $2$ whose fiber over $s$ consists of two distinct $k$-rational points $s_0,s_1$, such that $E(S' \setminus \{s_1\}) = E(S' \setminus \{s_0\}) = E(k)$.
\end{lemma}

\begin{proof}
	Let $E \subset \PP^2$ be given by some cubic $f \in k[x,y,z]$. The ring of $S$ is the integral domain $R = k[x,y,z]_{(x,y,z)} / (f)$. Since $2$ is invertible in $k$, a degree $2$ finite étale cover of $S$ may be constructed by adjoining to $R$ the square root of a unit $u \in R^\times$. Set $R' = R[t]/(t^2-u)$ and $S' = \Spec{R'}$. Then $\pi \colon S' \to S$ is split above $s$ if $u(s) \in \kappa(s)$ is a square. We choose $u = 1+x$.
	
	For the second part it suffices to prove that all $k$-morphisms $\alpha \colon S' \setminus \{s_1\} \to E$ are constant. Take a point $(a:b:c) \in E(k)$. Let $L \subset C$ be the corresponding ray, and $\eta \in S$ the generic point of $L$. As long as $a$ is non-zero, $u$ is not a square at $\eta$. Then the fiber of $\pi$ at $\eta$ is a single point $\eta'$ with rational function field. Therefore, $\alpha$ must send $\eta'$ to some closed point $p \in E$. By continuity we have $\alpha(s_0) = p$ as well. After passage to an algebraic closure of $k$, the points $\eta'$ as above lie dense in $S' \setminus \{s_1\}$. So $\alpha$ maps a dense subset of $S' \setminus \{s_1\}$ to $p$. Since $p$ is closed, $\alpha$ is constant.
\end{proof}

\begin{remark}
\label{rmk:char}
	The preceding lemma is still true in characteristic $2$, where one may construct a suitable cover by means of an Artin--Schreier extension. Therefore, the restriction on $k$ is not necessary; it is imposed only to simplify the exposition.
\end{remark}
Write $U = S \setminus \{s\}$, $U_0 = S' \setminus \{s_1\}$, $U_1 = S' \setminus \{s_0\}$, and $U_{01} = U_0 \cap U_1$. Then $\U = \{U_0,U_1\}$ is an open cover of $S'$. The associated first Čech cohomology is given by the Mayer--Vietoris exact sequence
\begin{equation}
\label{eq:cech}
	\begin{tikzcd}[column sep=small]
		0						\arrow{r}	&
		E(S')					\arrow{r}	&
		E(U_0) \times E(U_1)		\arrow{r}	&
		E(U_{01})				\arrow{r}	&
		\HH^1(\U,E)				\arrow{r}	&
		0.
	\end{tikzcd}
\end{equation}
By construction we have $E(S') = E(U_0) = E(U_1) = E(k)$. So \eqref{eq:cech} reduces to
\begin{displaymath}
	\begin{tikzcd}[column sep=small]
		0						\arrow{r}	&
		E(k)						\arrow{r}	&
		E(U_{01})				\arrow{r}	&
		\HH^1(\U,E)				\arrow{r}	&
		0.
	\end{tikzcd}
\end{displaymath}
The Galois group $G = \Aut(S'/S)$ acts on this sequence: the involution $\sigma \in G$ acts on $E(k)$ by inversion, on $E(U_{01})$ by $a \mapsto -\sigma^*a$, and on $\HH^1(\U,E)$ by $[X] \mapsto [\sigma^{-1}X]$. In fact this action comes from the natural $G$-action on \eqref{eq:cech}. Taking anti-invariants yields
\begin{displaymath}
	\begin{tikzcd}[column sep=small]
		0						\arrow{r}	&
		E(k)						\arrow{r}	&
		E(U)						\arrow{r}	&
		\HH^1(\U,E)^{-\sigma}
	\end{tikzcd}
\end{displaymath}
where for any $G$-module $M$ we denote by $M^{-\sigma} = \{m \in M : \sigma m = -m\}$ its subgroup of anti-invariants.

Let $\tilde{E} = S' \otimes_{\AUT(S'/S)} E_S$ be the $-1$-twist of $E_S$ along $\pi$. In other words, $\tilde{E}$ is the quotient of $S' \times E_S$ by $\AUT(S'/S)$, where $\sigma$ acts as $(x,a) \mapsto (\sigma x,-a)$. Then $\tilde{E}$ is an elliptic curve over $S$ with a canonical $S'$-isomorphism $E_{S'} \cong \tilde{E}_{S'}$.

\begin{lemma}
\label{lem:cohom}
	There is a natural map $\HH^1(S',E)^{-\sigma} \to \HH^1(S,\tilde{E})$ whose kernel is $2$-torsion.
\end{lemma}

\begin{proof}
	Let $A = \pi_* E_{S'}$ be the Weil restriction of $E_{S'}$ to $S$. By \cite[VIII 5.6]{bib:sga4-2} the pushforward map $\pi_* \colon \HH^1(S',E) \to \HH^1(S,A)$ is an isomorphism. It is also equivariant for the natural $G$-action on $\HH^1(S,A)$, so $\pi_*$ restricts to an isomorphism $\HH^1(S',E)^{-\sigma} \to \HH^1(S,A)^{-\sigma}$.
	
	We have $\tilde{E} = A^{-\sigma}$, or more precisely $\tilde{E}(T) = A(T)^{-\sigma}$ for all schemes $T/S$. Consider the map $\id - \sigma \colon A \to \tilde{E}$ and the inclusion $\tilde{E} \to A$. The induced composition
	\begin{displaymath}
		\begin{tikzcd}[column sep=small]
			\HH^1(S,A)				\arrow{r}	&
			\HH^1(S,\tilde{E})		\arrow{r}	&
			\HH^1(S,A).
		\end{tikzcd}
	\end{displaymath}
	is again $\id - \sigma$. Restricted to $\HH^1(S,A)^{-\sigma}$ this is simply multiplication by $2$. Hence the kernel of $\HH^1(S,A)^{-\sigma} \to \HH^1(S,\tilde{E})$ is $2$-torsion.
\end{proof}
The map $U \to C \setminus \{s\} \to E$ is an element of $E(U)$, no multiple of which is constant. Via the injection $E(U)/E(k) \to \HH^1(S',E)^{-\sigma}$ and lemma \ref{lem:cohom} we obtain a non-torsion class $\tau \in \HH^1(S,\tilde{E})$. The corresponding $\tilde{E}$-torsor on $S$ is not representable by theorem \ref{thm:raynaud}. As we have explained in section \ref{sec:torsors}, this proves theorem \ref{thm:genus_one}.

\section{Raynaud's approach}
\label{sec:raynaud}

In this section we briefly compare our construction with that by Raynaud in \cite[XIII 3.2]{bib:raynaud}. We sketch his approach.

Let $R$ be a discrete valuation ring in which $2$ is invertible. (As before, the characteristic condition is imposed only to simplify the exposition.) Let $E/R$ be an elliptic curve, and let $V \subset E$ be the complement of the zero section.

\begin{lemma}
\label{lem:normal}
	There exist a normal scheme $Z$ over $R$ and an $R$-morphism $f \colon V \to Z$ that is an isomorphism on the generic fiber and constant on the special fiber.
\end{lemma}

\begin{proof}
	(This proof is different from Raynaud's in \cite[XIII 3.2 b]{bib:raynaud}.) Embedding $E$ in the projective plane, $V$ is isomorphic to the spectrum of $R[x,y] / (y^2 - x^3 - ax^2 - bx -c)$ for suitable $a,b,c \in R$. Let $t \in R$ be a uniformizer. Let $Z$ be the spectrum of $R[u,v] / (v^2 - u^3 - t^2au^2 - t^4bu - t^6c)$ and define $f \colon V \to Z$ on rings by $u \mapsto t^2x$, $v \mapsto t^3y$. Then $Z$ is regular in codimension $1$, hence normal \cite[23.8]{bib:matsumura}. On the generic fiber, $t$ is a unit so $f$ is an isomorphism. On the special fiber $f$ is constant with image $(0,0)$.
\end{proof}
Let $s \in Z$ be the image of $f$ on the special fiber. Let $S$ be the localization of $Z$ at $s$. It is normal by construction.

We may choose $R$ such that it admits a connected finite étale cover $\Spec{R'} \to \Spec{R}$ of degree $2$ that is split on the special fiber. Set $S' = S_{R'}$, then $\pi \colon S' \to S$ is finite étale of degree $2$ as well. Let $s_0,s_1$ be the lifts of $s$ to $S'$. Write $U = S \setminus \{s\}$, $U_0 = S' \setminus \{s_1\}$, $U_1 = S' \setminus \{s_0\}$, and $U_{01} = U_0 \cap U_1$.

\begin{lemma}
\label{lem:maps}
	$E(U_0) = E(U_1) = E(R')$.
\end{lemma}

\begin{proof}
	Throughout the proof, primes indicate base change along $R \to R'$, e.g. $E' = E_{R'}$. By symmetry, it suffices to prove that all $R'$-morphisms $\alpha \colon U_0 \to E'$ factor over $\Spec{R'}$.
	
	 Note that $\alpha$ extends to an $R'$-morphism $W \to E'$ for some open $W \subseteq Z'$ containing $U_0$. Consider the $R'$-rational map
	\begin{displaymath}
		\begin{tikzcd}[column sep=small]
			h \colon E'			\arrow[dashed]{r}	&
			V'					\arrow{r}			&
			Z'					\arrow[dashed]{r}	&
			W					\arrow{r}			&
			E'.
		\end{tikzcd}
	\end{displaymath}
	On the generic fiber, $h$ is a rational map of elliptic curves over a field, hence extends to a morphism. Since $E'$ is the Néron model of its generic fiber \cite[1.2.8]{bib:neron}, $h$ actually extends to an $R'$-morphism $E' \to E'$. Note that $h$ is constant on the fiber over the closed point of $\Spec{R'}$ under $s_0$. Hence $h$ comes from $E(R')$ by rigidity \cite[6.1]{bib:git} because $\Spec{R'}$ is connected.
\end{proof}
Let $\U$ be the open cover $\{U_0,U_1\}$ of $S'$. The Mayer--Vietoris sequence \eqref{eq:cech} now reduces to
\begin{displaymath}
	\begin{tikzcd}[column sep=small]
		0						\arrow{r}	&
		E(R')					\arrow{r}	&
		E(U_{01})				\arrow{r}	&
		\HH^1(\U,E)				\arrow{r}	&
		0
	\end{tikzcd}
\end{displaymath}
and taking anti-invariants for the action of the involution $\sigma \in \Aut(S'/S)$ yields
\begin{displaymath}
	\begin{tikzcd}[column sep=small]
		0						\arrow{r}	&
		E(R)						\arrow{r}	&
		E(U)						\arrow{r}	&
		\HH^1(\U,E)^{-\sigma}.
	\end{tikzcd}
\end{displaymath}
Let $\tilde{E}$ be the $-1$-twist of $E_S$ along $\pi$. As in lemma \ref{lem:cohom} we have a map $\HH^1(S',E)^{-\sigma} \to \HH^1(S,\tilde{E})$ with $2$-torsion kernel.

Let $\eta$ be the generic point of $\Spec{R}$. We have a map $U_\eta \to Z_\eta \to E_\eta$ by inverting $f_\eta$. Since $U$ is normal of dimension $1$ and $E$ is proper, it extends uniquely to a map $U \to E$. This map is an element of $E(U)$, no multiple of which comes from $E(R)$. We obtain a non-torsion class $\tau \in \HH^1(S,\tilde{E})$. By theorem \ref{thm:raynaud} it corresponds to a non-representable $\tilde{E}$-torsor on $S$, proving theorem \ref{thm:genus_one}.

\begin{remark}
\label{rem:raynaud}
	It is worthwile to observe that Raynaud proves slightly less. He constructs a non-torsion element in $\HH^1(S',E)$ as above. Let $A = \pi_*E_{S'}$ be the Weil restriction of $E_{S'}$ to $S$. The pushforward map $\pi_* \colon \HH^1(S',E) \to \HH^1(S,A)$ is an isomorphism by \cite[VIII 5.6]{bib:sga4-2}, so we find a non-torsion element $\gamma \in \HH^1(S,A)$. Consider the short exact sequence
	\begin{displaymath}
		\begin{tikzcd}[column sep=small]
			0				\arrow{r}	&
			E_S				\arrow{r}	&
			A				\arrow{r}	&
			\tilde{E}		\arrow{r}	&
			0
		\end{tikzcd}
	\end{displaymath}
	of abelian schemes over $S$. In the long exact sequence of cohomology, either $\gamma$ maps to a non-torsion class in $\HH^1(S,\tilde{E})$, or $\gamma$ lifts to a non-torsion class in $\HH^1(S,E)$. This proves that there exists a non-representable torsor under either $\tilde{E}$ or $E_S$, hence theorem \ref{thm:genus_one}. However, this shorter proof does not permit us to explicitly write down a counterexample.
\end{remark}

\end{document}